\documentclass[a4paper,10pt]{article}
\usepackage{graphicx}
\usepackage{amsmath}
\usepackage{amssymb}
\usepackage{amsfonts}
\usepackage{amsthm}
\usepackage[all]{xy}

\newtheorem{theorem}{Theorem}[section]

\newtheorem{proposition}[theorem]{Proposition}

\theoremstyle{definition}
\newtheorem{definition}{Definition}[theorem]
\newtheorem{example}{Example}[definition]

\newtheorem{remark}{Remark}[definition]

\DeclareMathAlphabet{\mathpzc}{OT1}{pzc}{m}{it}
\begin{document}

\begin{center}

{ \LARGE \bfseries An infinite loop space machine for $\infty$-operads}\\[0.5cm]
\textsc{ \large  Gijs Heuts } \\[2cm]

\end{center}

\begin{abstract}
This paper describes a consequence of the more general results of \cite{heuts} which is of independent interest. We construct a functor from the category of dendroidal sets, which models the theory of $\infty$-operads, into the category of $E_\infty$-spaces. Applying May's infinite loop space machine for $E_\infty$-spaces then gives an infinite loop space machine for $\infty$-operads. We show that our machine exhibits the homotopy theory of $E_\infty$-spaces as a localization of the homotopy theory of $\infty$-operads.
\end{abstract}

\tableofcontents

\section*{Introduction}
In this paper we introduce a construction which associates an $E_\infty$-space to any dendroidal set $X$. From this one can then get an infinite loop space. This construction follows from the more general results of \cite{heuts}, which treats the theory of algebras over $\infty$-operads valued in either spaces or $\infty$-categories. Since the topic of an infinite loop space machine for $\infty$-operads seems of independent interest, this paper is a summary of the results relating to this machine. Most proofs however are not given here, but only in the more lengthy paper \cite{heuts}. \par 
The idea is as follows. We will first fix a cofibrant resolution of the terminal object in $\mathbf{dSets}$ (which is simply the nerve of the commutative operad), which we call $E_\infty$. The induced adjunction
\[
\xymatrix@R=40pt@C=40pt{
\mathbf{dSets}/E_\infty \ar@<.5ex>[r] & \mathbf{dSets} \ar@<.5ex>[l]
}
\]
is a Quillen equivalence \cite{moerdijkcisinski}. We then apply the so-called \emph{straightening functor}, which goes from $\mathbf{dSets}/E_\infty$ to the category of $E_\infty$-spaces. One can then apply May's infinite loop space machine \cite{may} to obtain an infinite loop space. The straightening functor we construct generalizes Lurie's straightening functor for simplicial sets \cite{htt}. \par 
We will describe a model structure on $\mathbf{dSets}/E_\infty$, which we call the \emph{covariant model structure}, for which the straightening functor in fact induces a Quillen equivalence between $\mathbf{dSets}/E_\infty$ and the category of $E_\infty$-spaces endowed with the projective model structure. We then show that the covariant model structure is a localization of the usual Cisinski-Moerdijk model structure on $\mathbf{dSets}/E_\infty$ constructed in \cite{moerdijkcisinski}. The covariant model structure can in fact be regarded as a generalization of the usual Kan-Quillen model structure on simplicial sets to the category of dendroidal sets. \par
The plan of this paper is as follows:
\begin{itemize}
\item Section 1 contains a brief discussion of necessary prerequisites concerning dendroidal sets. This discussion is by no means comprehensive, but contains pointers to references which are.
\item Section 2 will discuss left fibrations of dendroidal sets. For any dendroidal set $S$ we will define the covariant model structure on the slice category $\mathbf{dSets}/S$.
\item Section 3 discusses the construction of the straightening functor, which sends dendroidal sets over $S$ to algebras over the simplicial operad $\mathrm{hc}\tau_d(S)$ associated to $S$. In this paper we will only need the special case $S = E_\infty$.
\item Section 4 discusses our infinite loop space machine and contains further discussion of the covariant model structure. In particular, we show that the covariant model structure on $\mathbf{dSets}$ is a left Bousfield localization of the Cisinski-Moerdijk model structure.
\end{itemize}  

\subsection*{Acknowledgements}
I would like to thank Urs Schreiber and Dave Carchedi for some illuminating conversations. Many thanks are due to Ieke Moerdijk, who introduced me to dendroidal sets.

\section{Prerequisites}
Throughout this text the word operad will always mean a \emph{symmetric coloured operad}. An operad $P$ in a given closed symmetric monoidal category $\mathbf{\mathcal{E}}$ with tensor unit $I$ can be described by a set of colours $C$ and for any tuple of colours $(c_1, \ldots, c_n, c)$ an object
\begin{equation*}
P(c_1, \ldots, c_n; c)
\end{equation*} 
of $\mathbf{\mathcal{E}}$, which is to be thought of as the object of operations of $P$ with inputs $c_1, \ldots, c_n$ and output $c$. Furthermore, for $c \in C$, we should have a \emph{unit}
\begin{equation*}
I \longrightarrow P(c; c)
\end{equation*}
and we should have \emph{compositions}
\begin{equation*}
P(c_1, \ldots, c_n; c) \otimes P(d_1^1, \ldots, d_1^{j_1}; c_1) \otimes \ldots \otimes P(d_n^1, \ldots, d_n^{j_n}; c_n) \longrightarrow P(d_1^1, \ldots, d_n^{j_n}; c)
\end{equation*}
Finally, permutations $\sigma \in \Sigma_n$ should act by transformations
\begin{equation*}
\sigma^*: P(c_1, \ldots, c_n; c) \longrightarrow P(c_{\sigma(1)}, \ldots, c_{\sigma(n)}; c)
\end{equation*}
All of these data are required to satisfy various well-known associativity, unit and equivariance axioms. The case of interest in this paper is $\mathbf{\mathcal{E}} = \mathbf{sSets}$, the category of simplicial sets, in which case we obtain the definition of a \emph{simplicial operad}. We will denote the category of simplicial operads by $\mathbf{sOper}$. \par 
Any symmetric monoidal simplicial category $\mathbf{\mathcal{C}}$ gives rise to a simplicial operad $\underline{\mathbf{\mathcal{C}}}$ by
\begin{equation*}
\underline{\mathbf{\mathcal{C}}}(c_1, \ldots, c_n; c) := \mathbf{\mathcal{C}}(c_1 \otimes \cdots \otimes c_n, c)
\end{equation*}
Given a simplicial operad $P$, an algebra over $P$ in $\mathbf{\mathcal{C}}$ is then simply a morphism of operads
\begin{equation*}
P \longrightarrow \underline{\mathbf{\mathcal{C}}}
\end{equation*}
We briefly review the basics of dendroidal sets, of which an extensive treatment can be found in \cite{dendroidalsets} and \cite{moerdijkweiss2}, and fix our notation. The category $\mathbf{\Omega}$ is defined to be the category of finite rooted trees. These are trees equipped with a distinguished outer vertex called the \emph{output} and a (possibly empty) set of outer vertices not containing the output called \emph{inputs}. When drawing trees, we will always omit out- and input vertices from the picture. Recall that each such rooted tree $T$ defines a $\mathbf{Sets}$-operad $\Omega(T)$, the free operad generated by $T$, which is coloured by the edges of $T$. A morphism of trees $S \longrightarrow T$ is defined to be a morphism of operads $\Omega(S) \longrightarrow \Omega(T)$. The category of dendroidal sets is defined to be the category of presheaves on $\mathbf{\Omega}$:
\begin{equation*}
\mathbf{dSets} := \mathbf{Sets}^{\mathbf{\Omega}^{\mathrm{op}}}
\end{equation*} 
The dendroidal set represented by a tree $T$ will be denoted by $\Omega[T]$. We will denote the set of $T$-dendrices of a dendroidal set $X$ by $X_T$. There is an embedding of the simplex category into the category of finite rooted trees, denoted
\begin{equation*}
i: \mathbf{\Delta} \longrightarrow \mathbf{\Omega}
\end{equation*}
defined by sending $[n]$ to the linear tree $L_n$ with $n+1$ edges and $n$ inner vertices. By left Kan extension this induces an adjunction
\[
\xymatrix@R=40pt@C=40pt{
i_!: \mathbf{sSets} \ar@<.5ex>[r] & \mathbf{dSets}: i^*\ar@<.5ex>[l]
}
\]
As is the case in the simplex category, any morphism in $\Omega$ may be factorized into \emph{face and degeneracy maps}. Relations between these maps extending the well-known relations in the simplex category are described in $\cite{dendroidalsets}$. An inner face of $T$ contracting an inner edge $e$ will be denoted $\partial_e\Omega[T]$, an outer face chopping off a corolla with vertex $v$ is denoted $\partial_v\Omega[T]$. We let $F(T)$ denote the set of faces of $T$. For an inner edge we have the \emph{inner horn}
\begin{equation*}
\Lambda^e[T] := \bigcup_{\phi \in F(T) \backslash \partial_e\Omega[T]} \phi
\end{equation*}
and similarly the \emph{outer horn}
\begin{equation*}
\Lambda^v[T] := \bigcup_{\phi \in F(T) \backslash \partial_v\Omega[T]} \phi 
\end{equation*}
\par 
We will call a tree with one vertex and $n$ leaves an \emph{$n$-corolla}. The tree with no vertices and only a single edge will be denoted by $\eta$. We will often blur the distinction between $\Omega[\eta]$ and $\eta$ and write $\eta$ for the former, or $\eta_c$ if we want to be explicit about the fact that the unique edge of $\Omega[\eta]$ has colour $c$. Note that we have an isomorphism
\begin{equation*}
\mathbf{dSets}/\eta \simeq \mathbf{sSets}
\end{equation*} 
A dendroidal set $X$ is said to be an \emph{$\infty$-operad} if it has the extension property with respect to all inner horn inclusions of trees. If $X$ is an $\infty$-operad then the simplicial set $i^*(X)$ is an $\infty$-category and a 1-corolla of $X$ is called an \emph{equivalence} if the induced 1-simplex of $i^*(X)$ is an equivalence. \par 
The functor
\begin{equation*}
\mathbf{\Omega} \longrightarrow \mathbf{Oper_{Sets}}: T \longmapsto \Omega(T)
\end{equation*}
defines, by left Kan extension, an adjunction
\[
\xymatrix@R=40pt@C=40pt{
\tau_d: \mathbf{dSets} \ar@<.5ex>[r] & \mathbf{Oper_{Sets}}: N_d\ar@<.5ex>[l]
}
\]
The right adjoint $N_d$ is called the \emph{dendroidal nerve}. Recall that the category $\mathbf{Oper_{Sets}}$ carries a tensor product $\otimes_{BV}$ called the \emph{Boardman-Vogt tensor product}. For two representables $\Omega[S], \Omega[T] \in \mathbf{dSets}$ their tensor product is defined by
\begin{equation*}
\Omega[S] \otimes \Omega[T] := N_d(\Omega(S) \otimes_{BV} \Omega(T))
\end{equation*}
We extend this definition by colimits to all of $\mathbf{dSets}$. By general arguments the functor $- \otimes X$ has a right adjoint $\mathbf{Hom}_\mathbf{dSets}(X, -)$, making $\mathbf{dSets}$ into a closed symmetric monoidal category. \par 
The category of dendroidal sets is closely related to the category $\mathbf{sOper}$ of simplicial operads. The Boardman-Vogt $W$-construction with respect to the interval $\Delta^1 \in \mathbf{sSets}$ (see \cite{bergermoerdijk2} and \cite{dendroidalsets} for a detailed description) yields a functor
\begin{equation*}
\mathbf{\Omega} \longrightarrow \mathbf{sOper}: T \longmapsto W(\Omega(T))
\end{equation*}
By left Kan extension this gives an adjunction
\[
\xymatrix@R=40pt@C=40pt{
\mathrm{hc}\tau_d: \mathbf{dSets} \ar@<.5ex>[r] & \mathbf{sOper}: \mathrm{hc}N_d\ar@<.5ex>[l]
}
\]
The right adjoint $\mathrm{hc}N_d$ is called the \emph{homotopy coherent dendroidal nerve}. Let us describe the simplicial operad $W(\Omega(T))$ explicitly. Given colours $c_1, \ldots, c_n, c$ of $T$, we describe the space of operations $W(\Omega(T))(c_1, \ldots, c_n; c)$. Suppose there exists a subtree $S$ of $T$ such that the leaves of $S$ are $c_1, \ldots, c_n$ and its root is $c$. If it exists, such an $S$ is unique. Define
\begin{equation*}
W(\Omega(T))(c_1, \ldots, c_n; c) := (\Delta^1)^{I(S)}
\end{equation*}
where $I(S)$ denotes the set of inner edges of the tree $S$. If this set is empty the right-hand side is understood to be the point $\Delta^0$. If there exists no $S$ matching the description above, we let this space of operations be empty. Composition is defined by grafting trees, assigning length 1 to newly arising inner edges (i.e. the edges along which the grafting occurs).  
\par 
We will without explicit mention use basic facts from the theory of model categories, which can for example be found in \cite{hirschhorn} and \cite{hovey}. \par 
A monomorphism $f: X \longrightarrow Y$ of dendroidal sets is said to be \emph{normal} if, for any tree $T \in \mathbf{\Omega}$ and any $\alpha \in Y_T$ which is not in the image of $f$, the stabilizer $\mathrm{Aut}(T)_{\alpha}$ is trivial. A dendroidal set $X$ is called $\emph{normal}$ if the unique map $\emptyset \longrightarrow X$ is normal. The normal monomorphisms are the weakly saturated class generated by the boundary inclusions of trees $\partial\Omega[T] \subseteq \Omega[T]$. By Quillen's small object argument any map of dendroidal sets may be factored as a normal monomorphism followed by a morphism having the right lifting property with respect to all normal monomorphisms, which we refer to as a \emph{trivial fibration}. In particular, factoring a map $\emptyset \longrightarrow X$ in this way, we obtain a normal dendroidal set $X_{(n)}$ which we call a \emph{normalization} of $X$. An easy to prove and very useful fact is that any dendroidal set admitting a map to a normal dendroidal set is itself normal. \par 
A map $f: X \longrightarrow Y$ of dendroidal sets is called an $\emph{operadic equivalence}$ if there exists normalizations $X_{(n)}$ and $Y_{(n)}$ and a map $f_{(n)}: X_{(n)} \longrightarrow Y_{(n)}$ making the obvious diagram commute such that for any $\infty$-operad $Z$ the induced map
\begin{equation*}
i^* \mathbf{Hom_{dSets}}(Y_{(n)}, Z) \longrightarrow i^* \mathbf{Hom_{dSets}}(X_{(n)}, Z)
\end{equation*}
is a categorical equivalence of simplicial sets, i.e. an equivalence in the Joyal model structure. The following was established by Cisinski and Moerdijk in \cite{moerdijkcisinski}:
\begin{theorem}
There exists a combinatorial model structure on the category of dendroidal sets in which the cofibrations are the normal monomorphisms and the weak equivalences are the operadic equivalences. The fibrant objects of this model structure are precisely the $\infty$-operads. By slicing over $\eta$ we obtain a model structure on $\mathbf{sSets}$ which coincides with the Joyal model structure.
\end{theorem}

We will refer to this model structure as the \emph{Cisinski-Moerdijk model structure}. \par 

\section{Left fibrations and the covariant model structure}
In this section we will define \emph{left fibrations} of dendroidal sets, which are a generalization of left fibrations of simplicial sets. In the same way that left fibrations with codomain a fixed simplicial set $S$ model functors from $S$ into the $\infty$-category of spaces, our left fibrations of dendroidal sets with codomain a dendroidal set $S$ model $S$-algebras in the $\infty$-category of spaces. \par 

\begin{definition}
Let $p: X \longrightarrow S$ be a map of dendroidal sets. Then $p$ is a \emph{left fibration} if the following conditions are satisfied:
\begin{itemize}
\item $p$ is an inner fibration
\item For any corolla $\sigma$ of $S$ having inputs $\{s_1, \ldots, s_n\}$ (note that the set of inputs could be empty) and colors $\{x_1, \ldots, x_n\}$ of $X$ satisfying $p(x_i) = s_i$ for $1 \leq i \leq n$, there exists a corolla $\xi$ of $X$ with inputs $\{x_1, \ldots, x_n \}$ such that $p(\xi) = \sigma$
\item For any tree $T$ with at least two vertices and any leaf vertex $v$ of $T$, there exists a lift in any diagram of the form
\[
\xymatrix{
\Lambda^v[T] \ar[r]\ar[d] & X \ar[d]^p \\
\Omega[T] \ar[r]\ar@{-->}[ur] & S
}
\]
\end{itemize}
\end{definition}

\begin{remark}
\label{rmk:simplicialleftfib}
If $p: X \longrightarrow S$ is a left fibration, then the induced map $i^*p: i^*X \longrightarrow i^*S$ is a left fibration of simplicial sets, i.e. it has the right lifting property with respect to horn inclusions $\Lambda^n_i \longrightarrow \Delta^n$ for $0 \leq i < n$. In particular, if $i^*S$ is a Kan complex, then $i^*X$ is a Kan complex as well by a fundamental result of Joyal (see Proposition 1.2.5.1 of \cite{htt}). Also, if $q: K \longrightarrow L$ is a left fibration of simplicial sets, then $i_!(q)$ is a left fibration of dendroidal sets. 
\end{remark}

Our goal in this section is to describe a model structure on the category $\mathbf{dSets}/S$ in which the fibrant objects are precisely the left fibrations with codomain $S$. We will first endow $\mathbf{dSets}/S$ with the structure of a simplicial category.

\begin{definition}
Given maps $X \longrightarrow S$ and $Y \longrightarrow S$, we define the simplicial set $\mathrm{Map}_S(X, Y)$ as follows:
\begin{equation*}
\mathrm{Map}_S(X, Y)_n := \mathbf{dSets}/S(X \otimes i_!(\Delta^n), Y) 
\end{equation*}
where the map $X \otimes i_!(\Delta^n) \longrightarrow S$ is obtained by composing the projection map $X \otimes i_!(\Delta^n) \longrightarrow X$ with the map $X \longrightarrow S$.
\end{definition}

Note that $\mathrm{Map}_S(X, Y)$ satisfies the following universal property: for any simplicial set $K$ there is an isomorphism
\begin{equation*}
\mathbf{sSets}(K, \mathrm{Map}_S(X, Y)) \simeq \mathbf{dSets}/S(X \otimes i_!(K), Y) 
\end{equation*}
This isomorphism is natural in $K$. \par 
The mapping objects $\mathrm{Map}_S(X,Y)$ make $\mathbf{dSets}/S$ into a simpicial category. Also recall that $\mathbf{dSets}/S$ is tensored and cotensored over $\mathbf{sSets}$, a fact we already used in the definition above.

\begin{definition}
We will call a map in $\mathbf{dSets}/S$ a \emph{covariant cofibration} if its underlying map of dendroidal sets is a cofibration. We will call a map $f: X \longrightarrow Y$ in $\mathbf{dSets}/S$ a \emph{covariant equivalence} if for any left fibration $Z \longrightarrow S$ and normalizations $X_{(n)}$ and $Y_{(n)}$ of $X$ resp. $Y$ the induced map
\begin{equation*}
\mathrm{Map}_S(Y_{(n)}, Z) \longrightarrow \mathrm{Map}_S(X_{(n)}, Z)
\end{equation*}
is a weak homotopy equivalence of simplicial sets.
\end{definition}

\begin{remark}
The condition that there exist normalizations $X_{(n)}$ and $Y_{(n)}$ of $X$ resp. $Y$ such that the induced map
\begin{equation*}
\mathrm{Map}_S(Y_{(n)}, Z) \longrightarrow \mathrm{Map}_S(X_{(n)}, Z)
\end{equation*}
is a weak homotopy equivalence is equivalent to the condition that for any choice of normalizations  $X_{(n)}$ and $Y_{(n)}$ the stated map is a weak homotopy equivalence.
\end{remark}

The following result can be found in Section 6.2 of \cite{heuts}.

\begin{theorem}
\label{thm:covmodelstruct}
There exists a model structure on $\mathbf{dSets}/S$ in which the cofibrations (resp. weak equivalences) are the covariant cofibrations (resp. covariant weak equivalences). This model structure is combinatorial, left proper and simplicial. In this model structure the fibrant objects are precisely the left fibrations over $S$.
\end{theorem}

We will refer to this model structure as the \emph{covariant model structure}. It turns out that the weak equivalences between fibrant objects of $\mathbf{dSets}/S$ are easily characterized:

\begin{proposition}
Suppose we are given a diagram
\[
\xymatrix{
X \ar[rd]_p\ar[rr]^f & & Y \ar[dl]^q \\
& S & 
}
\]
such that $p$ and $q$ are left fibrations. Then the following are equivalent:
\begin{itemize}
\item[(i)] The map $f$ is a covariant equivalence
\item[(ii)] The map $f$ is an operadic equivalence
\item[(iii)] For every color $s \in S$ the induced map of fibers $f_s: X_s \longrightarrow Y_s$ is a weak homotopy equivalence simplicial sets
\end{itemize} 
\end{proposition}
\begin{remark}
The fibers appearing in condition (iii) are automatically Kan complexes. One sees this by noting that left fibrations are stable under pullback and invoking Remark \ref{rmk:simplicialleftfib}.
\end{remark}

\section{The straightening functor}

In this section we will describe the relation between the category $\mathbf{dSets}/S$ and the category of algebras over the simplicial operad $\mathrm{hc}\tau_d(S)$ under the assumption that $S$ is cofibrant in the Cisinski-Moerdijk model structure, i.e. normal. \par 
The results of Berger and Moerdijk \cite{bergermoerdijk1}\cite{bergermoerdijk2} in particular yield the following:
\begin{theorem}
If $S$ is normal, so that $\mathrm{hc}\tau_d(S)$ is cofibrant, there exists a left proper simplicial model structure on the simplicial category $\mathrm{Alg}_{\mathrm{hc}\tau_d(S)}(\mathbf{sSets})$ of simplicial $\mathrm{hc}\tau_d(S)$-algebras in which a map of algebras is a weak equivalence (resp. a fibration) if and only if it is a pointwise weak equivalence (resp. a pointwise fibration).
\end{theorem}

We will now define the so-called \emph{straightening functor}
\begin{equation*}
St_S: \mathbf{dSets}/S \longrightarrow \mathrm{Alg}_{\mathrm{hc}\tau_d(S)}(\mathbf{sSets})
\end{equation*}
Note that we can also describe $\mathbf{dSets}/S$ as a presheaf category; indeed, we have
\begin{equation*}
\mathbf{dSets}/S \simeq \mathbf{Sets}^{(\int_{\mathbf{\Omega}}S)^{\mathrm{op}}}
\end{equation*} 
where $\int_{\mathbf{\Omega}}S$ is the category of elements of $S$. From this we conclude that $\mathbf{dSets}/S$ is generated under colimits by objects of the form $\Omega[T] \longrightarrow S$ for $T \in \mathbf{\Omega}$. Since the straightening functor is supposed to be a left adjoint, it will suffice to construct it on these generators and then extend its definition by a left Kan extension. \par 
First, consider the special case where $S = \Omega[T]$ and $p$ is the identity map of $\Omega[T]$. For any color $c$ of $T$ let $T/c$ denote the subtree of $T$ which consists of $c$ and `everything above $c$'. 

\begin{example}
If $T$ is the tree
\[
\xymatrix@R=10pt@C=12pt{
&&&&&&\\
&*=0{\bullet}\ar@{-}[u]&&*=0{\bullet}&&&\\
&&*=0{\bullet}\ar@{-}[ul]\ar@{-}[ur]&&*=0{\bullet}\ar@{-}[u]&&\\
&&&*=0{\bullet}\ar@{-}\ar@{-}[ul]^{c}\ar@{-}[ur]&&&\\
&&&\ar@{-}[u]&&&\\
&&&&&&
}
\]
then $T/c$ is the tree
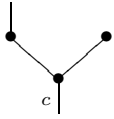
\[
\xymatrix@R=10pt@C=12pt{
&&&&&&\\
&*=0{\bullet}\ar@{-}[u]&&*=0{\bullet}&&&\\
&&*=0{\bullet}\ar@{-}[ul]\ar@{-}[ur]&&&\\
&&\ar@{-}[u]^c&&&&\\
&&&&&&
}
\]
\end{example} 

Define the cube
\begin{equation*}
\Delta[T/c] := (\Delta^1)^{\mathrm{col}(T/c)\backslash \{c\}}
\end{equation*}
where the product on the right is understood to be $\Delta^0$ if the set occurring in the exponent is empty.
The $\mathrm{hc}\tau_d(\Omega[T])$-algebra $St_{\Omega[T]}(\mathrm{id}_{\Omega[T]})$ is given by
\begin{equation*}
St_{\Omega[T]}(\mathrm{id}_{\Omega[T]})(c) := \Delta[T/c]
\end{equation*} 
The structure maps
\[
\xymatrix{
\mathrm{hc}\tau_d(\Omega[T])(c_1, \ldots, c_n; c) \times St_{\Omega[T]}(\mathrm{id}_{\Omega[T]})(c_1) \times \cdots \times St_{\Omega[T]}(\mathrm{id}_{\Omega[T]})(c_n) \ar[d] \\ St_{\Omega[T]}(\mathrm{id}_{\Omega[T]})(c)
}
\]
are given by grafting trees, assigning length 1 to the newly arising inner edges $c_1, \ldots, c_n$. 
\par 
Now let $S$ be any dendroidal set and $p: \Omega[T] \longrightarrow S$ a map. We get a map $\mathrm{hc}\tau_d(p)$ of simplicial operads, which induces an adjunction
\[
\xymatrix@R=40pt@C=40pt{
\mathrm{hc}\tau_d(p)_!: \mathrm{Alg}_{\mathrm{hc}\tau_d(\Omega[T])}(\mathbf{sSets}) \ar@<.5ex>[r] &  \mathrm{Alg}_{\mathrm{hc}\tau_d(S)}(\mathbf{sSets}): \mathrm{hc}\tau_d(p)^* \ar@<.5ex>[l]
}
\]
We set
\begin{equation*}
St_S(p) := \mathrm{hc}\tau_d(p)_!(St_{\Omega[T]}(\mathrm{id}_{\Omega[T]}))
\end{equation*}
Functoriality in $p$ is given as follows. Suppose we are given maps
\[
\xymatrix{
\Omega[R] \ar[r]^f & \Omega[T] \ar[r]^p & S
}
\]
If $f$ is a face map of $T$, the map $St_S(f)$ is described by the inclusions
\begin{equation*}
\Delta[S/c] \simeq \Delta[S/c] \times \{0\}^{\mathrm{col}(T/f(c))\backslash f(\mathrm{col}(S/c))} \longrightarrow \Delta[T/c]
\end{equation*}
for colours $c$ of $R$. If $f$ is a degeneracy, it is clear how to define $St_S(f)$. \par 
Having defined the functor $St_S$ on all maps of the form $\Omega[T] \longrightarrow S$, we take a left Kan extension of $St_S$ to all of $\mathbf{dSets}/S$ to obtain a functor
\begin{equation*}
St_S: \mathbf{dSets}/S \longrightarrow \mathrm{Alg}_{\mathrm{hc}\tau_d(S)}(\mathbf{sSets}): X \longmapsto \varinjlim_{\Omega[T] \rightarrow X} St_S(\Omega[T] \rightarrow X)
\end{equation*}
Since $St_S$ preserves colimits, the adjoint functor theorem provides us with a right adjoint to the straightening functor. We call this right adjoint the \emph{unstraightening functor} and denote it $Un_S$. One of the main results of \cite{heuts} is the following:

\begin{theorem}
\label{thm:straightening}
Let $S$ be a normal dendroidal set. Then the adjunction
\[
\xymatrix@R=40pt@C=40pt{
St_S: \mathbf{dSets}/S \ar@<.5ex>[r] & \mathrm{Alg}_{\mathrm{hc}\tau_d(S)}(\mathbf{sSets}): Un_S \ar@<.5ex>[l]
}
\]
is a Quillen equivalence.
\end{theorem}

\section{The covariant model structure on $\mathbf{dSets}$}
In the special case that $P$ is the terminal object $*$ of $\mathbf{dSets}$, which is isomorphic to $N_d(\mathcal{C}\mathit{omm})$, Theorem \ref{thm:covmodelstruct} gives us a simplicial model structure on $\mathbf{dSets}$ itself. In this section we will denote the category of dendroidal sets equipped with this model structure by $\mathbf{dSets}_{\mathrm{cov}}$ in order to avoid confusion with the usual Cisinski-Moerdijk model structure. The normalization $\nu: E_\infty \longrightarrow *$ induces a Quillen equivalence
\[
\xymatrix@R=40pt@C=40pt{
\nu_!: (\mathbf{dSets}/E_\infty)_{\mathrm{cov}} \ar@<.5ex>[r] & \mathbf{dSets}_{\mathrm{cov}}: \nu^* \ar@<.5ex>[l]
}
\]
The straightening functor and Theorem \ref{thm:straightening} give us a Quillen equivalence
\[
\xymatrix@R=40pt@C=40pt{
St_{E_\infty}: (\mathbf{dSets}/E_\infty)_{\mathrm{cov}} \ar@<.5ex>[r] & \mathrm{Alg}_{\mathrm{hc}\tau_d(E_\infty)}(\mathbf{sSets}): Un_{E_\infty} \ar@<.5ex>[l]
}
\]
One can then compose with May's infinite loop space machine for $E_\infty$-spaces \cite{may} to obtain a functor assigning an infinite loop space to a dendroidal set. \par
\begin{remark}
In \cite{heuts} we showed how to define symmetric monoidal $\infty$-categories as a certain type of fibration over $E_\infty$. Therefore our construction in particular provides an infinite loop space machine for symmetric monoidal $\infty$-categories.
\end{remark}

Our goal for the remainder of this section is to investigate $\mathbf{dSets}_{\mathrm{cov}}$ a little more closely and show that it is in fact a localization of $\mathbf{dSets}$ equipped with the Cisinski-Moerdijk model structure.

\begin{proposition}
The fibrant objects of $\mathbf{dSets}_{\mathrm{cov}}$ are precisely the dendroidal sets $X$ satisfying the following:
\begin{itemize}
\item[(i)] Let $n \geq 0$ and let $c_1, \ldots, c_n$ denote the colors of the leaves of $C_n$. Then $X$ has all fillers of the form
\[
\xymatrix{
\coprod_{i=1}^n \eta_{c_i} \ar[r]\ar[d] & X \\
\Omega[C_n] \ar@{-->}[ur] &
}
\]
\item[(ii)] If $T$ is a tree with at least 2 vertices and $\phi$ is any face of $T$ apart from a possible face chopping of the root, then any map $\Lambda^\phi[T] \longrightarrow X$ can be extended to $\Omega[T]$, i.e. $X$ has all horn fillers of the form
\[
\xymatrix{
\Lambda^\phi[T] \ar[r]\ar[d] & X \\
\Omega[T] \ar@{-->}[ur] &
}
\]
\item[(iii)] In case the root vertex $v$ of $T$ is unary, $X$ has horn fillers of the form
\[
\xymatrix{
\Lambda^v[T] \ar[r]\ar[d] & X \\
\Omega[T] \ar@{-->}[ur] &
}
\]
\end{itemize}
In particular, if $X$ is fibrant the simplicial set $i^*(X)$ is a Kan complex. 
\end{proposition}
\begin{proof}
By definition, fibrancy is equivalent to $(i)$ and $(ii)$, so we need to show that a fibrant object of $\mathbf{dSets}_{\mathrm{cov}}$ satisfies $(iii)$. Let $\alpha$ be a unary corolla of $X$. For any $n \geq 2$ there exists a lift in any diagram as follows:
\[
\xymatrix{
i_!(\Delta^1) \ar[d]_{\{0,1\}}\ar[dr]^\alpha & \\
i_!(\Lambda^n_0) \ar[d]\ar[r] & X \\
i_!(\Delta^n) \ar@{-->}[ur] & 
}
\]
By a fundamental lemma of Joyal (see Proposition 1.2.4.3 of \cite{htt}) we see that $\alpha$ is an equivalence in $X$. But then Theorem 4.2 of \cite{moerdijkcisinski} tells us that $X$ has property (iii) if the root corolla of $T$ is mapped to $\alpha$. Since $\alpha$ was arbitrary, the result follows. $\Box$
\end{proof}

\begin{definition}
We call a dendroidal set $X$ satisfying the conditions of the previous proposition a \emph{dendroidal Kan complex}. 
\end{definition}

\begin{remark}
If $X$ is isomorphic to $i_!(K)$ for a simplicial set $K$, then $X$ is a dendroidal Kan complex if and only if $K$ is a Kan complex in the usual sense. Kan complexes in the category of simplicial sets can be thought of as modelling $\infty$-groupoids. Similarly, our previous zigzag of Quillen equivalences between $\mathbf{dSets}_{\mathrm{cov}}$ and the category of $E_\infty$-spaces allows us to think of dendroidal Kan complexes as symmetric monoidal $\infty$-groupoids. 
\end{remark}

\begin{remark}
There is a canonical isomorphism of categories
\begin{equation*}
\mathbf{dSets}/\eta \simeq \mathbf{sSets}
\end{equation*}
Under this isomorphism the model structure on $\mathbf{dSets}/\eta$ induced by the covariant model structure coincides with the usual Kan-Quillen model structure on $\mathbf{sSets}$. Indeed, the cofibrations are the monomorphisms and the fibrant objects are precisely the Kan complexes. Therefore the covariant model structure on $\mathbf{dSets}$ can be regarded as a generalization of the Kan-Quillen model structure to dendroidal sets.
\end{remark}

\begin{proposition}
\label{prop:leftBousfield}
If $\mathbf{dSets}$ is equipped with the Cisinksi-Moerdijk model structure, then the identity functor
\begin{equation*}
\mathrm{id}: \mathbf{dSets} \longrightarrow \mathbf{dSets}_{\mathrm{cov}}
\end{equation*}
is a left Bousfield localization.
\end{proposition}
\begin{proof}
The two model structures under consideration have the same cofibrations, so it remains to prove that any operadic equivalence is a covariant equivalence. This will follow immediately from the fact that the Cisinski-Moerdijk model structure has more fibrant objects than the covariant model structure. Indeed, let $f: X \longrightarrow Y$ be an operadic equivalence and let $X_{(n)}$ and $Y_{(n)}$ be normalizations of $X$ and $Y$ respectively. Let $Z$ be any dendroidal Kan complex. Then it is in particular an $\infty$-operad, so that the map
\begin{equation*}
\mathrm{Map}(Y_{(n)}, Z) \longrightarrow \mathrm{Map}(X_{(n)}, Z)
\end{equation*}
is a weak homotopy equivalence by assumption. This proves $f$ is a covariant equivalence. $\Box$
\end{proof}

\begin{remark}
In fact, the same argument will prove something stronger; if $S$ is an $\infty$-operad, the covariant model structure on $\mathbf{dSets}/S$ is a left Bousfield localization of the model structure on this slice category induced by the Cisinski-Moerdijk model structure.
\end{remark}

It is straightforward to describe the left Bousfield localization of Proposition \ref{prop:leftBousfield} as a localization with respect to a family of cofibrations with cofibrant domain, which we will do now.

\begin{definition}
We define the set of \emph{generating left anodynes} to be the set consisting of the following maps:
\begin{itemize}
\item For any tree $T$ with at least two vertices and any leaf vertex $v$ of $T$, the map
\begin{equation*}
\Lambda^v[T] \longrightarrow \Omega[T]
\end{equation*}
\item For any $n \geq 0$, the map
\begin{equation*}
\coprod_{i=1}^n \eta_{c_i} \longrightarrow \Omega[C_n]
\end{equation*}
where $\{c_1, \ldots, c_n\}$ is the set of leaves of $C_n$
\end{itemize}
The weakly saturated class generated by the generating left anodynes is called the class of left anodyne morphisms, although we will not have to use that class in this paper.
\end{definition} 

\begin{proposition}
The left Bousfield localization of the Cisinski-Moerdijk model structure with respect to the set of generating left anodynes is the covariant model structure.
\end{proposition}
\begin{proof}
This follows easily from the fact that the fibrant objects in the covariant model structure are exactly the fibrants in the Cisinski-Moerdijk model structure which are local with respect to generating left anodynes. $\Box$ 
\end{proof}

\newpage
\bibliographystyle{plain}
\bibliography{biblio}

\end{document}